\documentclass{amsart}

 \newtheorem{theo}{Theorem}[section]
 \newtheorem{coro}[theo]{Corollary}
 \newtheorem{lemma}[theo]{Lemma}
 \newtheorem{prop}[theo]{Proposition}

\theoremstyle{definition}

\newtheorem*{case1}{Case 1}
\newtheorem*{case2}{Case 2}
\newtheorem*{case3}{Case 3}

\theoremstyle{remark}
\newtheorem{remark}[theo]{Remark}

\usepackage{graphicx}
\usepackage{amsmath}
\usepackage{amsfonts}
\usepackage{amssymb}
\usepackage[all]{xy}

\newcommand{\qf}[1]{\mbox{$\langle #1\rangle $}}
\newcommand{\pff}[1]{\mbox{$\langle\!\langle #1\rangle\!\rangle $}}
\newcommand{\qpf}[1]{\mbox{$\langle\!\langle #1]]$}}
\newcommand{\HH}{{\mathbb H}}

\newcommand{\NN}{{\mathbb N}}

\newcommand{\MM}{{\mathbb M}}

\renewcommand{\phi}{\varphi}

\DeclareMathOperator{\chr}{char}
\DeclareMathOperator{\ani}{an}

\DeclareMathOperator{\spa}{span}
\DeclareMathOperator{\Br}{Br}
\DeclareMathOperator{\ind}{ind}
\DeclareMathOperator{\rad}{Rad}

\begin{document}

\title[Witt kernels for quartic extensions]{Witt kernels and 
Brauer kernels for quartic extensions in characteristic two}

\author{Detlev W.~Hoffmann}
\email{detlev.hoffmann@math.tu-dortmund.de}

\author{Marco Sobiech}
\email{marco.sobiech@uni-dortmund.de}

\address{Fakult\"at f\"ur Mathematik,
Technische Universit\"at Dortmund,
D-44221 Dortmund,
Germany}

\thanks{The first author carried out part of this research during 
stays at the Fields Institute at Toronto in March and May 2013 in the frame
of the Thematic Program on {\it Torsors, Nonassociative Algebras and 
Cohomological Invariants}, and at the University of Ottawa and the University 
of Western Ontario in March 2013.  He is grateful to all these institutes for
their generous hospitality.\\
The first author's research on this paper is also supported in part by 
DFG project HO 4784/1-1 {\it Annihilators
and kernels in Kato's cohomology in positive characteristic and in
Witt groups in characteristic $2$}}

\begin{abstract}
Let $F$ be a field of characteristic $2$ and let $E/F$ be a field extension
of degree $4$.  We determine the kernel $W_q(E/F)$ of the
restriction map  $W_qF\to W_qE$  between the
Witt groups of nondegenerate quadratic forms over $F$ and over $E$,
completing earlier partial results by Ahmad, Baeza, Mammone and Moresi. 
We also deduct the corresponding result for the Witt kernel 
$W(E/F)$ of the 
restriction map  $WF\to WE$  between the
Witt rings of nondegenerate symmetric bilinear forms over $F$ and over $E$
from earlier results by the first author.  As application, we
describe the $2$-torsion part of the Brauer kernel for such extensions.  
\end{abstract}

\keywords{quartic extension, quadratic form, bilinear form, 
Witt group, Witt ring, Witt kernel, Brauer group, 
quaternion algebra, biquaternion algebra}

\subjclass[2010]{Primary 11E04; Secondary 11E81 12F05 16K50}

\maketitle

\section{Introduction}
When considering algebraic objects defined over fields such as quadratic
forms or central simple algebras, an important problem is to characterize 
their behavior under field extensions, for example to determine which
quadratic forms become hyperbolic or which central simple algebras split
over a given extension, in other words to compute the Witt kernel or
the Brauer kernel for that extension.    The purpose of the present paper
is to determine Witt kernels for quartic extensions in characteristic $2$
completing earlier results by various authors, and to apply this 
to the determination of the
$2$-torsion part of the Brauer kernel.  These kernels have been previously 
computed in characteristic not $2$.  Also, in all characteristics the 
kernels have been known for quite some time in the case of quadratic 
extensions, and they
are trivial for odd degree extensions due to Springer's theorem, so
quartic extensions in characteristic $2$ are the first case where the 
determination of these kernels was still incomplete.  We will survey
the known results in \S 2 and \S 4, including the bilinear case.

There are several aspects that complicate matters when studying
the Witt kernels for quartic extensions in characteristic $2$. 
Firstly, one has to distinguish between the Witt kernel for bilinear
forms and that for quadratic forms.  Secondly, one has to handle
carefully various cases of separability and inseparability when 
dealing with quartic extensions. For that reason, we provide a quick
survey of quartic extensions in characteristic $2$ in \S 3. 
Finally, when computing the
Witt kernels for quadratic forms in characteristic $2$, we will 
often have to deal with singular
quadratic forms, something that can be largely ignored 
in characteristic not $2$.  Accordingly, the formulation of the main
result in Theorem \ref{main} concerning quadratic Witt kernels
of simple quartic extensions is more complex than the corresponding 
result in characteristic not $2$.  Section 5  will be devoted to the
proof of that theorem and we will also show how previously known
results on generators for the quadratic Witt kernels for non 
purely inseparable biquadratic and purely inseparable
simple quartic extensions relate to our list of generators for these
kernels.  In \S 6 we apply our knowledge of Witt kernels to determine
the $2$-torsion part of the Brauer kernel for quartic extensions in
characteristic $2$.

\section{Basic definitions and facts}
We refer to \cite{ekm} and \cite{hl1} for any undefined terminology 
or any basic facts
about quadratic and bilinear forms especially in the case of
characteristic $2$ that we do not mention explicitly.
All quadratic resp. bilinear forms over a field $F$ are assumed to
be finite-dimensional, and bilinear forms are always assumed to be 
symmetric.  
Let $b=(b,V)$ be a bilinear form defined on an
$F$-vector space $V$.   $b$ is said to be nonsingular if 
for its radical one has $\rad(b)=\{ x\in V\,|\,b(x,V)\}=0$.
In the sequel, we will always assume bilinear forms to be nonsingular.
We have the usual notions of isometry $\cong$, orthogonal sum $\perp$
and tensor product $\otimes$ for bilinear forms.  
We define the value sets
$D_F(b)=\{ b(x,x)\,|\,x\in V\setminus\{ 0\}\}$, 
$D_F^*(b)=D_F(b)\cap F^*$, $D_F^0(b)=D_F(b)\cup \{ 0\}$.
$b$ is said to be isotropic if $D_F(b)=D_F^0(b)$, anisotropic otherwise, 
i.e. if $D_F(b)=D_F^*(b)$.
The $2$-dimensional isotropic bilinear form is called a 
metabolic plane in which case there is a basis such that the Gram
matrix is of shape
$\MM_a\cong\begin{pmatrix} 0 & 1 \\ 1 & a \end{pmatrix}$, $a\in F$.
$\MM_0$ is called a hyperbolic plane, and 
a metabolic resp. hyperbolic bilinear form is one that is 
isometric to an orthogonal sum 
of metabolic resp. hyperbolic planes.  
Any bilinear $b$ with $D_F^*(b)\neq \emptyset$ can be diagonalized.
This always holds in characteristic not $2$, and also in characteristic
$2$ provided $b$ is not hyperbolic.
We write $b\cong\qf{a_1,\ldots,a_n}_b$ for such a diagonalization.
Each bilinear form $b$ decomposes as $b\cong b_{\ani}\perp b_m$
with $b_{\ani}$ anisotropic and $b_m$ metabolic.  $b_{\ani}$ is
uniquely determined up to isometry.  Two nonsingular bilinear forms $b$, $b'$
are Witt equivalent, $b\sim b'$,  if $b_{\ani}\cong b'_{\ani}$.  The
Witt equivalence classes together with addition resp. multiplication
induced by orthogonal sum resp.
tensor product form the Witt ring $WF$ of bilinear forms,
or the bilinear Witt ring for short.

In characteristic not $2$ the theories of bilinear and quadratic forms
are the same, so let us from now on assume   $\chr(F)=2$.

Let $q$ be a quadratic form defined on the $F$-vector space $V$.
The associated bilinear form of $q$
is defined to be $b_q(x,y)=q(x+y)-q(x)-q(y)$.   Isometry, isotropy,
$D_F(q)$, $D_F^*(q)$ and $D_F^0(q)$ are defined in an analogous way to
the bilinear case.  If $(q,V)$, $(q',V')$ are quadratic forms, we say
that $q'$ dominates $q$, denoted by $q\prec q'$, if there exists
an injective linear map $t:V\to V'$ with $q(x)=q'(tx)$ for all $x\in V$.
  
$q$ is said to be nonsingular if the radical $\rad(q):=\rad(b_q)=\{ 0\}$,
totally singular if $\rad(q)=V$.  
$q$ is totally singular iff $q$
is a diagonal form $\sum_{i=1}^n a_ix_i^2$, and we write
$q\cong \qf{a_1,\ldots,a_n}$.  Note that in this case
$D_F^0(q)$ is a finite-dimensional $F^2$-vector space inside $F$.
Two totally singular quadratic forms $q$, $q'$ are isometric iff
$\dim q=\dim q'$ and $D_F^0(q)=D_F^0(q')$.
$q$ is nonsingular iff it is isometric
to an orthogonal sum of forms of type $[a,b]:=ax^2+xy+by^2$.
A nonsingular $2$-dimensional isotropic quadratic form is called
a hyperbolic plane $\HH \cong [0,0]$, and a hyperbolic quadratic form
is an orthogonal sum of hyperbolic planes.
Any quadratic form $q$ can be decomposed
in the following way (see, e.g., \cite[Prop.  2.4]{hl1}):
$$q\cong i\times\HH\perp\widetilde{q}_r\perp\widetilde{q}_s\perp 
j\times\qf{0}$$
with $\widetilde{q}_r$ nonsingular, $\widetilde{q}_s$ totally
singular, $\widetilde{q}_r\perp\widetilde{q}_s$ anisotropic.
The form $\widetilde{q}_r\perp\widetilde{q}_s$ is uniquely
determined up to isometry and is called the anisotropic part
of $q$, denoted by $q_{\ani}$. $i=i_W(q)$ is called the Witt index
of $q$, $j=i_d(q)$ the defect, and  $i_t(q)=i_W(q)+i_d(q)$ the total
index, which is the same as the dimension of a maximal totally
isotropic subspace of $q$.  
Note that $\widetilde{q}_s\perp j\times\qf{0}$ is just the
restriction $q|_{\rad(q)}$ of $q$ to its radical and it is 
therefore also uniquely determined, as is the nondefective part
$i\times\HH\perp\widetilde{q}_r\perp\widetilde{q}_s$ of $q$.

Two quadratic forms $q$, $q'$ are
called Witt equivalent, $q\sim q'$, if $q_{\ani}\cong q'_{\ani}$.
The Witt equivalence classes of nonsingular quadratic forms over
$F$ together with addition induced by the orthogonal sum 
form the quadratic Witt group 
$W_qF$ which becomes a $WF$-module in a natural way.  
Note that in particular $\qf{a}_b\otimes q\cong aq$.

Pfister forms will play an important role in our investigations.
An $n$-fold bilinear Pfister form (in arbitrary characteristic) 
is a nonsingular bilinear form isometric to a form of type
$\pff{a_1,\ldots,a_n}_b:=\qf{1,-a_1}_b\otimes\ldots\otimes\qf{1,-a_n}_b$,
$a_i\in F^*$.
An $(n+1)$-fold quadratic Pfister form in  characteristic $2$ is 
a nonsingular quadratic form isometric to a form of type
$\qpf{a_1,\ldots,a_n,c}:=\pff{a_1,\ldots,a_n}_b\otimes [1,c]$,
$a_i\in F^*$, $c\in F$.  Pfister forms are either anisotropic or
metabolic (bilinear case) resp. hyperbolic (quadratic case).  
They are also round forms, i.e. if $\pi$ is a bilinear or quadratic
Pfister form and $x\in F^*$, then $\pi\cong x\pi$ iff $x\in D_F^*(\pi)$.
An $n$-fold quasi-Pfister form $\widetilde{\pi}$
is a totally singular quadratic form such that there exists
an $n$-fold bilinear Pfister form $\pi\cong \pff{a_1,\ldots,a_n}_b$ 
with $\widetilde{\pi}(x)=\pi (x,x)$, in which case we
write $\widetilde{\pi}\cong \pff{a_1,\ldots,a_n}$.  We then have
$D_F^0(\widetilde{\pi})= D_F^0(\pi)=F^2(a_1,\ldots,a_n)$.
Quasi-Pfister forms are clearly also round.
A quadratic form $q$ is a Pfister neighbor of the quadratic or quasi-Pfister
form $\pi$ iff $\dim q>\frac{1}{2}\dim\pi$ and there exists 
an $x\in F^*$ with  $q\prec x\pi$.   In this case, $q$ is isotropic
iff $\pi$ is isotropic.

We put
$\wp(F)=\{ \lambda^2+\lambda\,|\,\lambda\in F\}$,
and for $b\in F$ we write $\wp^{-1}(b)$ to denote a root of
$X^2+X+b$ in some algebraic closure of $F$.   The Arf invariant
of a nonsingular quadratic form $q\cong[c_1,d_1]\perp\ldots\perp [c_n,d_n]$
is defined to be $\Delta(q)=\sum_{i=1}^nc_id_i\in F/\wp(F)$.  A
$2$-dimensional nonsingular form $q=[a,b]$ is isotropic
iff $\Delta(\phi)=ab=0\in F/\wp(F)$.   

By abuse of notation, we often use the
same symbol to denote a quadratic or bilinear form and its Witt class. 
 
If $E/F$ is a field extension and $\phi$ is a bilinear resp. 
quadratic form over $F$, we denote by $\phi_E=\phi\otimes E$ the form
obtained by scalar extension to $E$.  This gives rise to the
restriction homomorphisms $WF\to WE$ resp.   $W_qF\to W_qE$
with kernels  $W(E/F)$ resp.$W_q(E/F)$.  We call these
kernels the bilinear resp. quadratic Witt kernel of the extension
$E/F$.

We will often freely use the following results.

\begin{lemma}\label{subform}
Let $\phi$ and $\psi$ be quadratic forms over $F$ with $\phi$ nonsingular,
and let $\sigma\cong\qf{c_1,\ldots,c_n}$ be a totally singular quadratic form
over $F$.
\begin{enumerate}
\item[{\rm (i)}] For all $d_1,\ldots, d_n\in F$ one has
$[c_1,d_1]\perp\ldots\perp [c_n,d_n]\perp\sigma\sim\sigma$.
\item[{\rm (ii)}] If $\phi\prec\psi$ then there exists a quadratic
form $\tau$ over $F$ with $\psi\cong\phi\perp\tau$.
\item[{\rm (iii)}] If $\psi$ is nonsingular, then the following
are equivalent: 
\begin{enumerate}
\item[{\rm (a)}] $\phi\perp\sigma\prec\psi$;
\item[{\rm (b)}] there exists a nonsingular quadratic form $\tau$ over $F$ of
dimension $\dim\psi-\dim\phi - 2\dim\sigma$ and
$d_1,\ldots, d_n\in F$ with 
$$\psi\cong\phi\perp [c_1,d_1]\perp\ldots\perp [c_n,d_n]\perp\tau\ ;$$
\item[{\rm (c)}] 
there exists a nonsingular quadratic form $\tau$ over $F$ of
dimension $\dim\psi-\dim\phi - 2n$ such that
$\psi\perp\phi\perp\sigma \sim \tau\perp\sigma$.
\end{enumerate}
\end{enumerate}
\end{lemma}

\begin{proof} 
(i) follows from $[c,d]\perp\qf{c}\cong\HH\perp\qf{c}$, and
(ii) follows from general properties of nonsingular
subspaces (see, e.g., \cite[Prop.~7.22]{ekm}).  

(iii)  The
equivalence of (a) and (b)  is a special case
of a more general result \cite[Lemma 3.1]{hl1}.  

(b) implies (c)
by adding $\phi$ on both sides noting that 
$\phi\perp\phi=(\dim\phi)\HH\sim 0$.

To show that (c) implies (b), note that by
adding $\phi$ on both sides and comparing dimensions, we get
$$\psi\perp\sigma\cong n
\HH\perp\phi\perp\tau\perp\sigma\ .$$
Also, $\sigma\prec n\HH\cong [c_1,0]\perp\ldots\perp [c_n,0]$.  
By \cite[Lemma 3.9]{hl1}, there exists a nonsingular form
$\rho\cong [c_1,d_1]\perp\ldots\perp [c_n,d_n]$ for suitable
$d_i\in F$ such that
$$\psi\perp n\HH\cong n\HH\perp\phi\perp\tau\perp\rho$$
and thus $\psi\cong \phi\perp\tau\perp\rho$ as desired.
\end{proof}

An essential ingredient in our studies is the
behavior of quadratic forms under quadratic extensions as described
in by the following well known results.  

\begin{lemma}\label{quad}
Let $\phi$ be an anisotropic quadratic form over $F$ and let $E/F$ be
a quadratic extension.
\begin{enumerate}
\item[{\rm (i)}] {\rm (Baeza \cite[4.3]{b1}, 
Hoffmann-Laghribi \cite[Lemma 5.4]{hl2}.)}
If $E=F(\sqrt{a})$, $a\in F\setminus F^2$, 
then $\phi_E$ is isotropic iff there
exists $\lambda\in F^*$ such that $\lambda\qf{1,a}\prec\phi$.
\item[{\rm (ii)}] {\rm (Baeza \cite[V.4.2]{b2}.)}
If $E=F(\wp^{-1}(a))$, $a\in F\setminus\wp(F)$, 
then $\phi_E$ is isotropic iff there
exists $\lambda\in F^*$ and a quadratic form $\psi$ over $F$ with
$\phi\cong \lambda[1,a]\perp\psi$.
\end{enumerate}
\end{lemma}

As a consequence, one can readily show the following.
\begin{prop}\label{quadkern}
Let $\phi$ be an anisotropic nonsingular quadratic form over 
$F$ and let $E/F$ be a quadratic extension.
\begin{enumerate}
\item[{\rm (i)}] {\rm (Ahmad \cite[Cor.~2.8]{a1}, Baeza \cite[4.3]{b1}.)}
If $E=F(\sqrt{a})$, $a\in F\setminus F^2$, 
then $\phi_E$ is hyperbolic iff
there exists a nonsingular quadratic form $q$ over $F$ with
$\phi\cong\qf{1,a}_b\otimes q$.  In particular,
$W_q(E/F)=\qf{1,a}_b\otimes W_qF$.
\item[{\rm (ii)}] {\rm (Baeza \cite[V.4.11]{b2}.)}
If $E=F(\wp^{-1}(a))$, $a\in F\setminus\wp(F)$,
then $\phi_E$ is hyperbolic iff there
exists a bilinear form $b$ over $F$ such that $\phi\cong b\otimes [1,a]$.
In particular, $W_q(E/F)=WF\otimes [1,a]$.\
\end{enumerate}
\end{prop}

The following lemma is an essential ingredient in the determination
of the Witt kernel for quartic extensions.
\begin{lemma}\label{pfistermult}
Let $\tilde{\pi}$ be an $n$-fold bilinear Pfister form with associated
totally singular form $\pi$, $q$ a nonsingular
quadratic form and $\phi\cong\tilde{\pi}\otimes q$.  Let $x\in F^*$
and suppose that $\phi$ is anisotropic and $\pi\perp\qf{x}\prec \phi$.
Then $\pi\perp x\pi\prec\phi$.
\end{lemma}

\begin{proof}  Since $1\in D_F(\pi)$ we have $1\in D_F(\phi)$ and the
roundness of $\tilde{\pi}$ implies that $q$ can be chosen to represent
$1$ as well, i.e. we may assume $q\cong[1,a]\perp q'$ for suitable
$a\in F$ and nonsingular $q'$.  Now 
$\pi\perp\qf{x}\prec \phi$ implies 
$$\dim(\phi\perp\pi\perp\qf{x})_{\ani} =\dim\phi -2^n-1\ .$$
We have $\tilde{\pi}\otimes[1,a]\perp\pi\sim\pi$ and thus
$$\phi\perp\pi\perp\qf{x}\sim \tilde{\pi}\otimes q'\perp\pi\perp\qf{x}\ .$$
The form on the right hand side has dimension 
$\dim\phi -2^n+1>\dim\phi -2^n-1$ and
is therefore isotropic.  But $\tilde{\pi}\otimes q'\perp\pi$ is 
anisotropic as it is dominated by $\phi$.  Hence,
$x\in D_F^*(\tilde{\pi}\otimes q'\perp\pi)$ and there exist
$u\in D_F^0(\tilde{\pi}\otimes q')$, $v\in D_F^0(\pi)$ with $x=u+v$.

If $u=0$ then $x=v\in D_F^*(\pi)$ and $\pi\perp\qf{x}$ would be isotropic,
a contradiction.  Thus, $u\neq 0$ and we may assume, again by the roundness
of Pfister forms, $q'\cong [u,w]\perp q''$ for some $w\in F$ and nonsingular
$q''$ and we see that 
$$\pi\perp u\pi\prec \tilde{\pi}\otimes ([1,a]\perp [u,w]\perp q'')\cong\phi\ .$$ 
Now suppose $\pi\cong\pff{r_1,\ldots, r_n}$, $r_i\in F$.
Then $\pi\perp x\pi\cong\pff{r_1,\ldots, r_n,x}$ and
$\pi\perp u\pi\cong\pff{r_1,\ldots, r_n,u}$.
Since $v\in D_F^0(\pi)=F^2(r_1,\ldots, r_n)$ and with $u=x+v$, we get
$$D_F^0(\pi\perp u\pi)=F^2(r_1,\ldots, r_n,u)=F^2(r_1,\ldots, r_n,x)
=D_F^0(\pi\perp ux\pi)$$ 
and thus
$\pi\perp x\pi\cong \pi\perp u\pi$, hence $\pi\perp x\pi\prec\phi$.
\end{proof}

\section{Quartic field extensions in characteristic $2$}
In this section we recall well known and some perhaps lesser 
known facts about quartic extensions in characteristic $2$ which
we will need later on.  We omit 
proofs as these results belong to basic field and Galois theory.

Let $F$ be a field of characteristic $2$ and let $E/F$ be a field extension
with $[E:F]=4$.    Such extensions can be classified as follows.

\subsection{Simple extensions of degree $4$}\label{simple} 
In this case,
let $E=F(\alpha)$ and let $f(X)=X^4+aX^3+bX^2+cX+d\in F[X]$ be
the minimal polynomial of $\alpha$ over $F$.  We
distinguish $4$ (sub)cases.

\begin{case1}
$E/F$ is separable.  This is the case iff $a\neq 0$ or $c\neq 0$.
We may assume (possibly after replacing $\alpha$ by $\alpha^{-1}$ and/or a
linear change of variables) that the minimal polynomial
of $\alpha$ is of shape $X^4+aX^3+cX+d\in F[X]$ with $a\neq 0$.
\end{case1}

\begin{case2}
$E/F$ is inseparable but not purely inseparable.  This is the case
iff the minimal polynomial is of shape $f(X)=X^4+bX^2+d$ with $b\neq 0$.
Note that the separable closure of $F$ inside $E$ is then the
(uniquely determined) separable quadratic subextension $F(\alpha^2)$.

We now consider the extension $F(\sqrt{b},\sqrt{d})/F$.
Since $f$ is irreducible, it is not possible that both $b,d\in F^2$.
Hence $[F(\sqrt{b},\sqrt{d}):F]=2$ or $4$ and we distinguish the respective
subcases.
\begin{enumerate}
\item[2a.]  $[F(\sqrt{b},\sqrt{d}):F]=2$.  This is the case iff
the inseparable closure of $F$ in $E$ is a (uniquely determined) 
inseparable quadratic extension $F(\sqrt{c})$ for a certain $c\in F\setminus F^2$.
In this case, one therefore has that $E=F(\sqrt{c},\alpha^2)$ is
``mixed'' biquadratic.  In field theory, such an algebraic extension
is often called balanced as it is the compositum of its
maximal inseparable and maximal separable subextensions.

One can furthermore show that this subcase holds iff
the equation $x^2+by^2+dz^2=0$ has
a nontrivial solution $(x,y,z)\neq (0,0,0)$ with $x,y,z\in F$, i.e.
iff the totally singular quadratic form $\qf{1,b,d}$ is isotropic over $F$.
We will refer to this case also as the mixed biquadratic case.
\item[2b.]  $[F(\sqrt{b},\sqrt{d}):F]=4$.  This is the case iff
the inseparable closure of $F$ in $E$ is just $F$ itself, i.e.
$E/F$ does \emph{not} contain an inseparable quadratic subextension.
So in particular, this extension is unbalanced.
This subcase holds iff $\qf{1,b,d}$ is anisotropic over $F$.
\end{enumerate}
\end{case2}

\begin{case3}  $E/F$ is purely inseparable. This is the case
iff the minimal polynomial is of shape $f(X)=X^4+d$, so
$E=F(\sqrt[4]{d})$ for some $d\in F\setminus F^2$.
\end{case3} 

\subsection{Nonsimple extensions of degree $4$}\label{nonsimple}    
$E/F$ is nonsimple
iff $E/F$ is biquadratic
purely inseparable:  $E=F(\sqrt{a},\sqrt{b})$ for suitable $a,b\in F^*$.

\subsection{The cubic resolvent of a polynomial of degree $4$}\label{cubic}
Let $f(X)=X^4+aX^3+bX^2+cX+d\in F[X]$ and let $\alpha_i$, $1\leq i\leq 4$
be the roots of $f$ in an algebraic closure of $F$. The $\alpha_i$
need not be distinct.  Let now
$$\beta_1 =  \alpha_1\alpha_2+\alpha_3\alpha_4\ ,\quad
\beta_2  =  \alpha_1\alpha_3+\alpha_2\alpha_4\ ,\quad
\beta_3  =  \alpha_1\alpha_4+\alpha_2\alpha_3\ .$$
The cubic resolvent of $f$ (in any characteristic) is then given by
$$f_C(X)=\prod_{i=1}^3(X-\beta_i)=X^3-bX^2+(ac-4d)X-(a^2d+c^2-4bd)\in F[X]\ .$$
To avoid confusion, let us mention that
in the literature, one often finds an alternative version of the
cubic resolvent where the $\beta_i$ are replaced by
$\gamma_1=(\alpha_1+\alpha_2)(\alpha_3+\alpha_4)$,
$\gamma_2=(\alpha_1+\alpha_3)(\alpha_2+\alpha_4)$,
$\gamma_3=(\alpha_1+\alpha_4)(\alpha_2+\alpha_3)$.  
Note that $\beta_i+\gamma_i=b$, so one readily gets 
that the cubic resolvent $\widetilde{f}_C$ in this version is given by
$$\widetilde{f}_C(X)=\prod_{i=1}^3(X-\gamma_i)=
X^3-2bX^2+(b^2+ac-4d)X+(c^2+a^2d-abc)=-f_C(-X+b)$$
For our purposes, we will need $f_C$.

Since we are in characteristic $2$, the formula simplifies:
$$f_C(X)=X^3+bX^2+acX+(a^2d+c^2)\ .$$
Now for a simple degree $4$ extension $E=F(\alpha)$ in characteristic $2$
where $\alpha$ has 
minimal polynomial $f(X)$ as described in the above cases 1--3, we get
the following corresponding cubic resolvents:
\begin{enumerate}
\item[1.] Minimal polynomial $f(X)=X^4+aX^3+cX+d$:  
$f_C(X)=X^3+acX+(a^2d+c^2)$.
\item[2.] Minimal polynomial $f(X)=X^4+bX^2+d$:  
$f_C(X)=X^3+bX^2$.
\item[3.] Minimal polynomial $f(X)=X^4+d$:  
$f_C(X)=X^3$.
\end{enumerate}

\section{Results on Witt kernels for finite field extensions}

One 
has $W(E/F)=0$ resp. $W_q(E/F)=0$ for odd degree extensions 
$E/F$ since anisotropic forms stay
anisotropic over odd degree extensions, a result often referred to as
Springer's theorem \cite{sp} but that has apparently been proved earlier 
by E.~Artin in a communication to E.~Witt (1937).  

Let us now assume that $\chr(F)\neq 2$. It is well known 
that if $E=F(\sqrt{d})$ is a quadratic extension, then $W(E/F)$ is generated 
by the norm form $\qf{1,-d}$ of that extension. The case $[E:F]=4$ is 
considerably more difficult to treat.  For biquadratic
extensions $E=F(\sqrt{a},\sqrt{b})$ it was shown by 
Elman-Lam-Wadsworth \cite{elw1} that $W(E/F)$ is generated
by $\qf{1,-a}$ and $\qf{1,-b}$, 
and the case of degree $4$ extensions containing a quadratic subextension
can be found in Lam-Leep-Tignol \cite{llt}.
  
The complete determination of Witt kernels for arbitrary degree $4$ 
extensions is due to Sivatski \cite{si}.  To formulate his result,
first note that in characteristic not $2$ the degree $4$ extension $E/F$
will be separable and hence simple, say, $E=F(\alpha)$.  Let
$f(X)\in K[X]$ be the minimal polynomial of $\alpha$.  We may assume 
(after a linear change of variables) that $f(X)=X^4+bX^2+cX+d$.
Recall from subsection \ref{cubic} that then
$\widetilde{f}_C(X)=X^3-2bX^2+(b^2-4d)X+c^2$

\begin{theo}{\rm (Sivatski \cite[Cor.~2, Cor.~4]{si}.)}  Let $F$ be a 
field of characteristic not $2$ and $E/F$ a
field extension of degree $4$.  Let $E=F(\alpha)$ be such that
$f(X)=X^4+bX^2+cX+d\in K[X]$ is the minimal polynomial of $\alpha$.
Then $W(E/F)$ is generated by $1$-fold Pfister forms 
$\pff{D}$ for all $D\in F^*\setminus F^{*2}$ with
$F(\sqrt{D})\subseteq E$, and $2$-fold Pfister forms
$\pff{\widetilde{f}_C(r),-r}$ with $r\in F^*$ such that 
$\widetilde{f}_C(r)\neq 0$.
\end{theo} 
 
Little is known for Witt kernels of higher even degree extensions.  
For triquadratic extensions $E=F(\sqrt{a_1},\sqrt{a_2},\sqrt{a_3})$ it is shown
by Elman-Lam-Tignol-Wadsworth \cite{eltw} that generally
$\sum_{i=1}^3\qf{1,-a_i}W(F)\subsetneq W(E/F)$, but an explicit description
of generators of $W(E/F)$ seems not to be known.  
Under strong assumptions on the base field, more can be said.  
For example, for local or
global fields, it is known that 
$W(F(\sqrt{a_1},\ldots,\sqrt{a_n})/F)=\sum_{i=1}^n\qf{1,-a_i}W(F)$, see
Elman-Lam-Wadsworth \cite{elw2}.

Let us now consider Witt kernels $W_q(E/F)$ for quadratic forms in 
characteristic $2$. For quadratic extensions $E/F$, see 
Proposition \ref{quadkern}.  The following summarizes the known results 
for separable biquadratic extensions and for multiquadratic extensions
of separability degree at most $2$.

\begin{theo}\label{multiquad} Let $F$ be a field of characteristic $2$, let
$\alpha_1,\ldots,\alpha_n,\beta_1,\beta_n$ be nonzero elements in 
an algebraic closure of $F$ such that there exist 
$a_1,\ldots,a_n$, $b_1,b_2\in F^*$ with $\alpha_i^2=a_i$ ($1\leq i\leq n$)
and $\beta_i^2+\beta_i=b_i$ ($i=1,2$).  Let $E=F(\alpha_1,\ldots,\alpha_n)$,
$E'=F(\alpha_1,\ldots,\alpha_n,\beta_1)$, $E''=F(\beta_1,\beta_2)$.
\begin{enumerate}
\item[{\rm (i)}] {\rm (Aravire-Laghribi \cite{al}; Mammone-Moresi \cite{mm}
for $n=2$.)} 
$$W_q(E/F)=\sum_{i=1}^n\qf{1,a_i}_b\otimes W_q(F)\ .$$
\item[{\rm (ii)}] {\rm (Aravire-Laghribi \cite{al}; Ahmad \cite{a2} for $n=1$)}
$$W_q(E'/F)=\sum_{i=1}^n\qf{1,a_i}_b\otimes W_q(F)+W(F)\otimes [1,b_1]\ .$$
\item[{\rm (iii)}] {\rm (Baeza \cite[4.16]{b2}.)} 
$$W_q(E''/F)=W(F)[1,b_1]+W(F)[1,b_2]\ .$$
\end{enumerate}
\end{theo}
For a different proof of parts (i) and (ii), see \cite{h2}.

For simple totally inseparable degree $4$ extensions, i.e. Case 3 in
section \ref{simple}, Ahmad has shown the following.
\begin{theo}{\rm (Ahmad \cite{a3}.)}\label{simpleinsep} 
Let $F$ be a field of characteristic $2$, let 
$\alpha$ be a nonzero element in 
an algebraic closure of $F$ such that there exists
$a\in F\setminus F^2$ with $\alpha^4=a$.  Let $E=F(\alpha)$.   
Then $W_q(E/F)$
is generated by $2$-fold quadratic Pfister forms of type
$\qpf{a,x}$ and $\qpf{x,ax^2y^2}$ with $x\in F^*$ and $y\in F^2(a)^*$.
\end{theo}

Again, not much else is known for other types of finite
algebraic extensions in characteristic $2$.

Finally, consider Witt kernels $W(E/F)$ for 
Witt rings of bilinear forms in characteristic
$2$. It is not difficult to show that if $E/F$ is separable, then $W(E/F)=0$ 
(Knebusch \cite{kn1}).
In \cite{h2}, Witt kernels for a large class of purely 
inseparable algebraic extensions
have been determined.  In particular, the following was shown.

\begin{theo}{\rm (Hoffmann \cite{h1}.)} Let
$E$ be a purely inseparable extension 
of exponent $1$
over a field $F$ of characteristic $2$ (i.e. $E^2\subset F\subset E$).
Then $W(E/F)$ is generated by bilinear forms
$\qf{1,t}_b$ where $t\in E^{*2}$. 
\end{theo}

Now the nonsimple degree $4$ extensions $E$ of a field $F$ of 
characteristic $2$ are exactly the biquadratic purely inseparable
extensions $E=F(\sqrt{a_1},\sqrt{a_2})$ and thus they are exactly
the purely inseparable exponent $1$ extensions of degree $4$, so 
this case is covered by the previous theorem.

The case of simple degree $4$ extensions  
follows from a much more general result
shown in \cite{dh}, where (in characteristic $2$) $W(E/F)$ was determined  
for \emph{arbitrary} function fields of hypersurfaces, i.e. 
for extensions $E/F$ where $E$ is the quotient
field of $F[X]/(f(X))$ for an irreducible polynomial 
$f(X)=f(X_1,\ldots,X_n)\in F[X_1,\ldots,X_n]$.  In the case of a simple
extension of even degree, this result boils down to the following.

\begin{theo}\label{bilinearkernel}  
Let $E=F(\alpha)$ be a simple extension of even degree $n$
of a field $F$ of characteristic $2$, and let
$X^n+a_{n-1}X^{n-1}+\ldots +a_1X+a_0\in F[X]$ be the minimal polynomial 
of $\alpha$.
\begin{enumerate}
\item[{\rm (i)}] If $E/F$ is separable (i.e. $a_i\neq 0$ for some odd $i$), then
$W(E/F)=0$.
\item[{\rm (ii)}] If $E/F$ is inseparable (i.e. $a_i=0$ for all odd $i$), 
let $K= F(\sqrt{a_0},\ldots ,\sqrt{a_{n-2}})$.
Then $[K:F]=2^s$ for some $1\leq s\leq \frac{n}{2}$, and
$W(E/F)$ is generated by the $s$-fold bilinear Pfister forms
$\pff{b_1,\ldots,b_s}_b$ with $b_i\in F^*$ such that
$F(\sqrt{b_1},\ldots,\sqrt{b_s})=K$.

More precisely, any anisotropic bilinear form $\phi$ over $F$
with $\phi\in W(E/F)$ can be written as
$\phi\cong\lambda_1\pi_1\perp\ldots\perp\lambda_n\pi_n$ for
suitable $n\in\NN$, $\lambda_i\in F^*$ and where the $\pi_i$ are 
$s$-fold bilinear Pfister forms of the above type.
\end{enumerate}
\end{theo}

\begin{proof}
Part (i) of this theorem follows from Knebusch's result \cite{kn1}
mentioned above.  Part (ii)
follows from \cite[Cor.~11.4]{dh}.
\end{proof}

\section{Witt kernels of simple quartic extensions}
We start with a proposition that is essentially 
due to Sivatski \cite[Prop.~1]{si}.
He assumed the characteristic to be different from $2$ but the proof
goes through basically without any changes also in characteristic 
$2$ and also for singular quadratic forms.  Our formulation
is slightly different. 
\begin{prop}\label{sivatski}   Let $m$ be a positive integer.
Let $f(X)\in F[X]$ be an irreducible
polynomial of even degree $2m$, $\alpha$ be a root
of $F$ in some algebraic closure of $F$ and $E=F(\alpha)$.  
Let $\phi$ be an anisotropic quadratic form over $F$ and assume
that over $E$, one has 
$i_t(\phi_E)>\bigl(\frac{1}{2}-\frac{1}{2m}\bigr)\dim\phi$.  Then
there exists a form $\phi_0\prec\phi$ such that 
$2\leq\dim\phi_0\leq m+1$ and $\phi_0$ represents $\lambda f(X)$ for
some $\lambda\in F^*$ over $F[X]$.  In particular, $(\phi_0)_E$ is isotropic.  
\end{prop}

\begin{coro}\label{sivatskicor}  
Let $E/F$ be a simple field extension with $[E:F]=4$ and
let $\phi$ be an anisotropic form in $W_q(E/F)$.  Then
there exists a form $\phi_0\prec\phi$ with $2\leq\dim\phi_0\leq 3$
and such that $(\phi_0)_E$ is isotropic.
\end{coro} 

From now on, let $F$ be a field of characteristic $2$ and
$E=F(\alpha)$ be a simple degree $4$ extension where $\alpha$ has minimal
polynomial $f(X)=X^4+aX^3+bX^2+cX+d$ with cubic resolvent
$f_C(X)=X^3+bX^2+acX+a^2d+c^2$.

\begin{lemma}\label{cubicpfister}
Let $e\in F^*$ be such that $f_C(e)\neq 0$.  Then
$\qpf{f_C(e),\frac{d}{e^2}}\in W(E/F)$.
\end{lemma}
\begin{proof}  Note that $\phi:=[1,\frac{d}{e^2}]\perp \qf{f_C(e)}$ is
a Pfister neighbor of $\qpf{f_C(e),\frac{d}{e^2}}$.  Thus,
$\qpf{f_C(e),\frac{d}{e^2}}_E$ is hyperbolic iff 
$\phi_E$ is isotropic iff the polynomial equation
$$\phi(X,Y,Z)=X^2+XY+\frac{d}{e^2}Y^2+(e^3+be^2+ace+a^2d+c^2)Z^2=0$$
has a nontrivial solution over $E$.  
Now put 
$$(X,Y,Z)=(e\alpha^2+c\alpha, e(a\alpha+e),\alpha)\neq (0,0,0)$$
By substituting we get
$$\begin{array}{rcl}
\phi(e\alpha^2+c\alpha, e(a\alpha+e),\alpha) & = &
e^2\alpha^4+c^2\alpha^2+(e\alpha^2+c\alpha)e(a\alpha+e)
\\
 & & +da^2\alpha^2+de^2+(e^3+be^2+ace+a^2d+c^2)\alpha^2\\[1ex]
 & = & e^2(\alpha^4+a\alpha^3+b\alpha^2+c\alpha+d)\\[1ex]
 & = & e^2f(\alpha)=0
\end{array}$$
as desired.
\end{proof} 

By subsection \ref{simple}, we may from now on assume the following
regarding the minimal polynomial $f(X)=X^4+aX^3+bX^2+cX+d$.
\begin{itemize}
\item $b=0$ and $a\neq 0$ if $E/F$ is separable (case 1 in 
subsection \ref{simple});
\item $a=c=0$, $b\neq 0$ if $E/F$ inseparable but not purely
inseparable (case 2 in subsection \ref{simple});
\item $a=b=c=0$ if $E/F$ is purely inseparable 
(case 3 in subsection \ref{simple}).
\end{itemize}
\begin{theo}\label{main}
$W_q(E/F)$ is generated as $WF$-module by
\begin{enumerate}
\item[(a)] $[1,g]$ for those $g\in F$ such that $F(\wp^{-1}(g))\subset E$;
\item[(b)]$\qpf{g,h}$ for $h\in F$ and those $g\in F^*$ with 
$F(\sqrt{g})\subset E$;
\item[(c)] $\qpf{f_C(e),\frac{d}{e^2}}$ for $e\in F^*$ with $f_C(e)\neq 0$;
\item[(d)] (only in case 2) $\qpf{b,d,h}$ with $h\in F$.
\end{enumerate}
\end{theo}

\begin{remark}\label{mainrem}  (i) Note that $[1,g]\neq 0\in W_q(F)$ iff $g\notin \wp(F)$ iff
$[F(\wp^{-1}(g)):F]=2$, in which case $F(\wp^{-1}(g))/F$ is a separable
quadratic extension.  Thus, such nonhyperbolic binary forms cannot show
up in the list of generators in $W_q(E/F)$ in case 3 
(purely inseparable extensions).

(ii) Simlarly, if $\qpf{g,h}$ is such that $F(\sqrt{g})\subset E$, then
in case 1 (separable extensions) this would imply $g\in F^{*2}$ and thus
$\qpf{g,h}=0\in W_q(F)$, so nonhyperbolic forms of this type cannot show up
 in the list of generators in $W_q(E/F)$ in case 1.

(iii) For forms of type $\qpf{b,d,h}$ with $h\in F$ in case $2$, note that
in the mixed biquadratic case 2a 
we have $[F(\sqrt{b},\sqrt{d}):F]=2$ which is equivalent
to $\qf{1,b,d}_b$ being isotropic by subsection \ref{simple}, hence
 $\qpf{b,d,h}$ is isotropic and hence $\qpf{b,d,h}=0\in W_q(F)$ in that case.
So nonhyperbolic forms of this type cannot show up
in the list of generators in $W_q(E/F)$ in case 2a.
\end{remark}

\begin{proof}[Proof of Theorem \ref{main}]
We first show that the above generators are indeed in $W_q(E/F)$.

\medskip

Generators of type (a):
If $F(\wp^{-1}(g))\subset E$ for $g\in F$, then $g\in\wp(E)$, hence
$\Delta([1,g])=g=0\in E/\wp(E)$ and thus $[1,g]_E=0\in W_q(E)$.

\medskip

Generators of type (b):  If $F(\sqrt{g})\subset E$ for $g\in F^*$ then
$g\in E^{*2}$, thus $(\qf{1,g}_b)_E$ is isotropic and hence metabolic,
so $(\pff{g}_b)_E=0\in WE$ and thus $\qpf{g,h}_E=0\in W_q(E)$.

\medskip

Generators of type (c):  This is Lemma \ref{cubicpfister}

\medskip 

Generators of type (d):  Here, by the previous remark, we may assume that
we are in case 2b with minimal polynomial $f(X)=X^4+bX^2+d$ with
$[F(\sqrt{b},\sqrt{d}):F]=4$.  By Theorem \ref{bilinearkernel} we have
$\pff{b,d}_b\in W(E/F)$ and thus $\qpf{b,d,h}\in W_q(E/F)$.

\medskip

We now prove that any form in $W_q(E/F)$ can be written as a sum
of scalar multiples of these generators.
Let $q$ be a form with $q\in W_q(E/F)$.  We use induction
on $\dim q=n$.  We may assume without loss of generality that $q$
is anisotropic  

If $n=2$, then we may assume after scaling that $q\cong [1,g]$ for
some $g\in F\setminus \wp(F)$.  Since $q_E=0\in W_q(E)$, we have
$g\in\wp(E)$ and it follows that $F(\wp^{-1}(g))\subset E$ and
therefore $q$ is of type (a).

\medskip

Now assume $n\geq 4$.  By Corollary \ref{sivatskicor}, there
exists a form $q_0\prec q$ with $\dim q_0\in\{ 2,3\}$ and $(q_0)_E$
isotropic.  If $\dim q_0=2$, we may have either (after scaling)
$q_0\cong [1,g]$ for some $g\in F\setminus \wp(F)$, or 
$q_0\cong\qf{1,g}$ for some $g\in F^*\setminus F^{*2}$.

If $q_0\cong [1,g]$, then $q_0$ is a generator of type (a) (as in case $n=2$),
and by Lemma \ref{subform} there exists a nonsingular form
$q'$ over $F$ with $q\cong q_0\perp q'$.  But $q, q'\in W_q(E)$, hence
also $q'\in W_q(E/F)$, and since $\dim q'=\dim q -2$, we are done
by induction.

If $q_0\cong \qf{1,g}$, then $(q_0)_E$ being isotropic is equivalent
to $g\in E^{*2}$, hence $F(\sqrt{g})\subset E$.  Furthermore,
by Lemma \ref{subform}, there exist $u,v\in F$ and a nonsingular form
$q'$ over $F$ with $q\cong [1,u]\perp g[1,v]\perp q'$.
Consider $\pi:=\qpf{g,v}$.  This is a generator of type (b), and 
in $W_q(F)$ we have 
$$\pi + q =[1,v]\perp g[1,v]\perp [1,u]\perp g[1,v]\perp q'
=[1,u+v]\perp q'\ .$$
Now $[1,u+v]\perp q'\in W_q(E/F)$ since $q,\pi\in W_q(E/F)$,
and $\dim([1,u+v]\perp q')=\dim(q)-2$.  Again, we are done
by induction.

\medskip

From now on we may therefore in addition assume that there is no
$2$-dimensional form $\phi\prec q$ with $\phi_E$ isotropic.
In particular, by Lemma \ref{quad}, $q$ will not become
isotropic over any quadratic intermediate extension 
$F\subset K\subset E$.

So let $q_0\prec q$ with $(q_0)_E$ isotropic and $\dim q_0=3$.
By Proposition \ref{sivatski}, $q_0$ represents a scalar
multiple of $f(X)$ over $F[X]$.  After
scaling $q$, we may assume that $q_0$ represents $f(X)$.
Let $U$ be the underlying $F$-vector space of $q_0$, and
let $B$ be the bilinear form associated with $q_0$.
The anisotropy of $q_0$ and a simple degree argument 
(using $\deg(f)=4$) show that there exist $u,v,w\in U$ such that
$q_0(uX^2+vX+w)=f(X)$, and therefore
$$\begin{array}{l}
q_0(u)X^4+B(u,v)X^3+(q_0(v)+B(u,w))X^2+B(v,w)X+q_0(w)\\[1ex]
=X^4+aX^3+bX^2+cX+d\ .\end{array}$$
Comparing coefficients yields
$$q_0(u)=1,\ B(u,v)=a,\ q_0(v)+B(u,w)=b,\ B(v,w)=c,\ q_0(w)=d\ .$$
Note also that $q_0(u\alpha^2+v\alpha+w)=f(\alpha)=0$, thus,
if $U'=\spa (u,v,w)$, we see that the form $q_0|_{U'}$ becomes
isotropic over $E$.  Since we assumed that no $2$-dimensional form
over $F$ dominated by $q$ becomes isotropic over $E$, we thus 
necessarily have that $U=U'$ and
$u,v,w$ are linearly independent.  In particular, 
$e:=q_0(v)\neq 0$ because $q_0$ is anisotropic. 
Writing the form $q_0$ as a 
homogeneous degree $2$ polynomial in three variables, we get
$$q_0(X,Y,Z)=X^2+aXY+eY^2+(e+b)XZ+cYZ+dZ^2\ .$$
We now make a further case distinction according to the types
of field extension.  We first treat the cases 1 and 3 before treating
the more difficult cases 2a and 2b.

\medskip

Cases 1 and 3  (separable and purely inseparable
extension): Here, we have $b=0$ (in the separable case we assumed
this without loss of generality), and we get
$$q_0(X,Y,Z)=X^2+aXY+eY^2+eXZ+cYZ+dZ^2\ .$$
We perform an invertible linear change of variables and obtain
$$q_0(X+cY,eY,\frac{1}{e}Z+aY)= 
X^2+(e^3+ace+da^2+c^2)Y^2+XZ+\frac{d}{e^2}Z^2$$
and therefore 
$${\textstyle q_0\cong [1,\frac{d}{e^2}]\perp\qf{f_C(e)}
\prec\qpf{f_C(e),\frac{d}{e^2}}\ .}$$
But then $\pi=\qpf{f_C(e),\frac{d}{e^2}}$ is a generator of type (c).
Since $q_0\prec q$, again by Lemma \ref{subform}, we can write
$q\cong [1,\frac{d}{e^2}]\perp f_C(e)[1,t]\perp q'$ for some
$t\in F$ and a nonsingular form $q'$.  But then, in $W_q(F)$,
$$\begin{array}{rcl}
\pi+q & = & [1,\frac{d}{e^2}]\perp f_C(e)[1,\frac{d}{e^2}]\perp
[1,\frac{d}{e^2}]\perp f_C(e)[1,t]\perp q'\\[1ex]
 & = & \underbrace{{\textstyle f_C(e)[1,t+\frac{d}{e^2}]\perp q'}}_{q''}\ .
\end{array}$$
Since $\pi, q\in W_q(E/F)$ we have $q''\in W_q(E/F)$, and
also $\dim(q'')=\dim q-2$ and we are done by induction.

\medskip

Case 2 (inseparable but not purely inseparable extension):
Here, we have $a=c=0$ and $b\neq 0$ and we get
$$q_0(X,Y,Z)=X^2+eY^2+(e+b)XZ+dZ^2\ .$$
First, assume $e'=e+b\neq 0$.  Then
$$q_0(X,e'Y,\frac{1}{e'}Z)=X^2+(e'^3+be'^2)Y^2+XZ+\frac{d}{e'^2}Z^2$$
and thus $q_0\cong [1,\frac{d}{e'^2}]\perp \qf{f_C(e')}$
and we can conclude as before.

Now suppose $e+b=0$, i.e. $e=b$. Then $q_0\cong\qf{1,b,d}$ 
is totally singular. The anisotropy of $q_0$ then
implies that we must be in case 2b.  
After scaling $q$ by $b$, we may assume $bq_0\cong\qf{1,b,bd}\prec q$.

Recall that $E=F(\alpha)$
with $\alpha^4+b\alpha^2+d=0$.  Put $\beta=\alpha^2$.  Then
$L=F(\beta)$ is the unique quadratic intermediate extension
$F\subset L\subset E$, and we have $L=F(\wp^{-1}(\frac{d}{b^2}))$
and $E=L(\sqrt{\beta})$.

Recall also, that by an earlier assumption, $q$ will be anisotropic over $L$
but hyperbolic over $E=L(\sqrt{\beta})$.
By Proposition \ref{quadkern}, there exists a nonsingular form $\psi$ over $L$ such that
$q_L\cong\pff{\beta}_b\otimes \psi$.    Note also that
$d=\beta b+\beta^2$ and thus
$$\qf{1,b,bd}_L\cong \qf{1,b,b(\beta b+\beta^2)}\cong\qf{1,\beta,b}
\prec q_L$$
where the last isometry holds since $\{ 1,b,b(\beta b+\beta^2)\}$ and
$\{ 1,b,\beta\}$ generate the same $L^2$-vector space inside $L$.
By Lemma \ref{pfistermult}, we then have $\pff{\beta,b}\prec q_L$. Also,
$L^2(\beta, b)=L^2(\beta b+\beta^2,b)=L^2(d,b)$ and therefore
$\pff{\beta,b}\cong\pff{b,d}_L$ and
$\pff{b,d}_L\prec q_L$.

Suppose first that $\pff{b,d}\prec q$.  Then there exist
$x,y,z,t\in F$ and a nonsingular form $q'$ over $F$ such that
$$q\cong [1,x]\perp b[1,y]\perp d[1,z]\perp bd[1,t]\perp q'$$
and thus, in $WF$, 
$$q +\qpf{b,d,x}=
\underbrace{b[1,x+y]\perp d[1,x+z]\perp bd[1,x+t]\perp q'}_{q''}$$
Since $\qpf{b,d,x}$ is a generator of type (d), we have 
$\qpf{b,d,x}, q\in W_q(E/F)$
and thus also $q''\in W_q(E/F)$.  But $\dim(q'')=\dim(q)-2$ and we are done
by induction.

Now suppose that that $\pff{b,d}\not\prec q$.   Together with
$\qf{1,b,bd}\prec q$, 
$\pff{b,d}_L\prec q_L$ and $q_L$ anisotropic,
we get that $i_W(q\perp\pff{b,d})=3$ and $i_W(q_L\perp\pff{b,d}_L)=4$.
Hence, there exists a nonsingular form $q'$ over $F$ with 
\begin{itemize}
\item $q\perp\pff{b,d}\cong 3\HH\perp q'\perp\pff{b,d}$,
\item $q'\perp \pff{b,d}$ anisotropic,
\item $(q'\perp\pff{b,d})_L$ isotropic, and 
\end{itemize}
By Lemma \ref{quad} applied to $L=F(\wp^{-1}(\frac{d}{b^2}))$, 
there exists $\lambda\in F^*$ such that
$\lambda[1,\frac{d}{b^2}]\prec q'\perp\pff{b,d}$, and hence there
exists a nonsingular form $q''$ over $F$ of dimension $\dim q'-2=\dim q-8$
with $q'\perp\pff{b,d}\cong q''\perp \lambda[1,\frac{d}{b^2}]\perp\pff{b,d}$.
We get by comparing dimensions
$${\textstyle q\perp \lambda[1,\frac{d}{b^2}]\perp\pff{b,d}\cong  
5\HH\perp q''\perp\pff{b,d}\ ,}$$
so $q\perp \lambda[1,\frac{d}{b^2}]$ cannot be anisotropic because then
$i_W(q\perp \lambda[1,\frac{d}{b^2}]\perp\pff{b,d})\leq\dim\pff{b,d}=4$,
a contradiction.
On the other hand $\lambda[1,\frac{d}{b^2}]\not\prec q$ because
otherwise $q_L$ would be isotropic, again a contradiction to an
earlier assumption.  As a consequence, 
$i_W(q\perp \lambda[1,\frac{d}{b^2}])=1$.  

Put
$\widehat{q}\cong (q\perp \lambda[1,\frac{d}{b^2}])_{\ani}$.  We then
have $\dim\widehat{q}=\dim q$.  Since $q\in W_q(E/F)$ and since
$[1,\frac{d}{b^2}]$ is a generator of type (a), it follows that
$\widehat{q}\in W_q(E/F)$.  Furthermore, the above shows that
$i_W(\widehat{q}\perp\pff{b,d})=4$ and thus $\pff{b,d}\prec \widehat{q}$.
By the same reasoning as before, we are done by induction if we replace
$q$ by $\widehat{q}$.  But since $q$ and $\widehat{q}$ only differ
by a scalar multiple of a generator of type (a), we are done.
\end{proof}

The next corollary will be crucial for determining Brauer 
kernels in the next section.

\begin{coro}\label{2-fold}
Let $\phi$ be an anisotropic nonsingular quadratic form  over
$F$ with $\phi\in W_q(E/F)$.  Suppose $\dim\phi= 4$ or $\dim\phi >4$
and  $\dim\phi = 2\bmod 4$.  Then there exists a $2$-fold quadratic 
Pfister form $\pi\in W_q(E/F)$ and a $3$-dimensional
Pfister neighbor $\psi$ of $\pi$ with $\psi\prec\phi$.
In particular, the $2$-fold Pfister forms in $W_q(E/F)$ are
exactly the ones of the following types:
\begin{enumerate}
\item[{\rm (a')}] $\qpf{h,g}$ for $h\in F^*$ and 
those $g\in F$ such that $F(\wp^{-1}(g))\subset E$;
\item[{\rm (b)}]$\qpf{g,h}$ for $h\in F$ and those $g\in F^*$ with 
$F(\sqrt{g})\subset E$;
\item[{\rm (c)}] $\qpf{f_C(e),\frac{d}{e^2}}$ 
for $e\in F^*$ with $f_C(e)\neq 0$.
\end{enumerate}
\end{coro}

\begin{proof}
If $\phi$ becomes isotropic already
over a quadratic subextension in $E$, then by Lemma \ref{quad} we 
may assume that after scaling $\qf{1,g}\prec\phi$ with
$F(\sqrt{g})\subset E$, or $[1,g]\prec\phi$ with 
$F(\wp^{-1}(g))\subset E$.  It is then clear that one can find
an $h\in F$ such that in the first case 
$\psi\cong [1,h]\perp\qf{g}\prec\phi$, and then $\psi$ is a Pfister
neighbor of $\qpf{g,h}$, and such that in the second case
$\psi\cong [1,g]\perp\qf{h} \prec\phi$, and then $\psi$ is a Pfister
neighbor of $\qpf{h,g}$.

So we may assume that $\phi$ stays anisotropic over any quadratic
subextension in $E$.  The proof of Theorem \ref{main} then shows
that there exists a $3$-dimensional form $\psi\prec\phi$ such that
either $\psi$ is a Pfister neighbor of 
$\qpf{f_C(e),\frac{d}{e^2}}$ for some $e\in F^*$ with $f_C(e)\neq 0$,
or $\psi$ is totally singular. But the proof also showed that 
if $\psi$ is totally singular then this
can only hold in the case 2b.  In that situation and with 
$E=F(\alpha)$ as in the proof and $L=F(\beta)$ with $\beta=\alpha^2$,
it was shown that $\phi_L\cong\pff{\beta}\otimes \tau$ for some
nonsingular form $\tau$ over $L$.  In particular, $\dim\phi\equiv 0\bmod 4$.
By assumption, this requires $\dim\phi =4$, but this is impossible
as $\phi$ is nonsingular and can therefore not dominate a 
$3$-dimensional totally singular form.

Now if $\phi\in W_q(E/F)$ is a $2$-fold quadratic Pfister form, then
by the above there exists a $2$-fold quadratic Pfister form  $\pi$
of type (a'), (b) or (c) and a $3$-dimensional Pfister neighbor
$\psi$ of $\pi$ with $\psi\prec\phi$.  After scaling $\phi$
(and thus also $\psi$), 
we may assume $\psi\prec\pi$, hence
$\dim(\phi\perp\pi)_{\ani}\leq 2$.  But 
$\Delta (\phi\perp\pi)=0\in F/\wp(F)$,
thus $\dim(\phi\perp\pi)_{\ani}=0$, therefore $\phi\cong\pi$.
\end{proof}

A $6$-dimensional nonsingular quadratic form with trivial
Arf invariant is called an Albert form.  We will need them in 
the determination of the $2$-torsion part of the
Brauer kernel $\Br_2(E/F)$ in section \ref{brauer}

\begin{coro}\label{albert}
Let $\phi\in W_q(E/F)$ be anisotropic with $\dim\phi=6$ and 
$\Delta(\phi)=0\in F/\wp(F)$.
Then there exist $2$-fold quadratic Pfister forms $\pi_1$, $\pi_2$
as in Corollary \ref{2-fold} (a'), (b), (c) and $\lambda\in F^*$ such that
$\phi = \lambda (\pi_1+\pi_2)\in W_qF$.
\end{coro} 

\begin{proof} After scaling and by Corollary \ref{2-fold}, there exists
such a $2$-fold quadratic Pfister form $\pi_1$ and a 
$3$-dimensional $\psi\prec\pi_1$ with
$\psi\prec\phi$.  Hence, there exists a nonsingular 
$\rho$ with $\dim\rho=4$ and  $\phi\perp\pi_1 =\rho\in W_qF$.
But then $\Delta(\rho)=0\in F/\wp(F)$, so there exists
$\lambda\in F$ such that $\rho\cong \lambda\pi_2$ for a $2$-fold quadratic
Pfister form $\pi_2$.  Since 
$0=(\phi\perp\pi_1)_E=(\lambda\pi_2)_E\in W_qE$, 
it follows that $\pi_2$ is as in Corollary \ref{2-fold} (a'), 
(b) or (c).  Also $\phi =\pi_1\perp\lambda\pi_2\in W_qF$.  
Comparing dimensions yields that $\pi_1\perp\lambda\pi_2$ is isotropic,
so $D_F(\pi_1)^*\cap D_F(\lambda\pi_2)^*\neq\emptyset$.
The roundness of $\pi_i$ then shows that one may assume without loss
of generality that $\lambda$ is
represented by $\pi_1$ and thus $\phi =\lambda(\pi_1\perp\pi_2)\in W_qF$.
\end{proof}

We now would like to align the result in Theorem \ref{main} with the
descriptions of generators in Theorem \ref{multiquad} for separable
and mixed biquadratic extensions (cases 1 and 2a) and 
in Theorem \ref{simpleinsep} for simple
purely inseparable degree $4$ extensions (case 3).
To do so, we use the following relations in the Witt group $W_q(F)$ that
are obvious or can be readily verified:
\begin{itemize} 
\item $[r,s]+[u,v]=[r+u,s]+[u,s+v]$.  In particular
$[1,u]+[1,v]=[1,u+v]$ and thus $\qpf{w,u}+\qpf{w,v}=\qpf{w,u+v}$;
\item $[1,u+v^2]=[1,u+v]$, in particular $[1,v+v^2]=0$;
\item $\qpf{u,u}=0$ and thus also $\qpf{u,u+v}=\qpf{u,v}$
and $\qpf{u+v,v}=\qpf{u+v,u}$;
\item $\qpf{uv,w}=\qpf{u,w}+u\qpf{v,w}$ and thus
$\qpf{uv,v}=\qpf{u,v}$;
\item $\qpf{uv^2,w}=\qpf{u,w}$, in particular $\qpf{v^2,w}=0$.
\end{itemize}

\subsection{Generators in the separable biquadratic case}
Let us assume that the degree $4$ extension $E/F$ is 
separable biquadratic, so $E=F(\mu,\nu)$ with
$\mu^2+\mu+u=0$ and $\nu^2+\nu+v=0$ for some $u,v\in F^*$.
By Theorem \ref{multiquad}(iii), $W_q(E/F)$ is generated (as 
$WF$-module) by $[1,u]$ and $[1,v]$.  In our list of
generators, there are forms $[1,w]$ with $F(\wp^{-1}(w))\subset E$.
But the only proper quadratic subextensions are 
$F(\mu)=F(\wp^{-1}(u))$, $F(\nu)=F(\wp^{-1}(v))$, and
$F(\mu+\nu)=F(\wp^{-1}(u+v))$, thus in the list of generators
in Theorem \ref{main}(a), the only anisotropic forms up to isometry
are $[1,u]$, $[1,v]$, $[1,u+v]=[1,u]+[1,v]$.

By Remark \ref{mainrem}(ii), there are no nonhyperbolic 
generators of type (b).  So the last type to consider are 
generators $\qpf{f_C(e),\frac{d}{e^2}}$ of type (c).  To do so,
let us first write $E$ as a simple extension.  Put
$\gamma:=\mu\nu +u+v$.  One checks that $E=F(\gamma)$ and that
the minimal polynomial of $\gamma$ is
$$f(X)=X^4+X^3+(u^2+v^2+uv)X+u^2v+uv^2+u^2v^2+u^4+v^4$$
with cubic resolvent
$$f_C(X)=X^3+(u^2+v^2+uv)X+u^2v+uv^2=(X+u)(X+v)(X+u+v)$$
The $X^0$ term of $f(X)$ is $d=u^2v+uv^2+u^2v^2+u^4+v^4$
and we get in $W_q(F)$:
$$\begin{array}{rcl}
\bigl[1,\frac{d}{e^2}\bigr] & = & \bigl[1,\frac{u^2v+v^2u}{e^2}+
\bigl(\frac{uv+u^2+v^2}{e}\bigr)^2\bigr]\\[1ex]
& = & \bigl[1,\frac{u^2v+v^2u}{e^2}+
\frac{uv+u^2+v^2}{e}\bigr]\\[1ex]
& = & \bigl[1,\frac{f_C(e)}{e^2}+e\bigr]
\end{array}$$
and hence, with $r=e+u$, $s=e+v$, $t=e+u+v$ and the above relations:
$$\begin{array}{rcl}
\qpf{f_C(e),\frac{d}{e^2}} & = & \qpf{\frac{f_C(e)}{e^2},\frac{f_C(e)}{e^2}+e}
= \qpf{\frac{f_C(e)}{e^2},e}\\[1ex]
 & = & \qpf{f_C(e),e}=\qpf{(e+u)(e+v)(e+u+v),e}\\[1ex]
 & = & \qpf{e+u,e}+(e+u)\qpf{e+v,e}+(e+u)(e+v)\qpf{e+u+v,e}\\[1ex]
 & = & \qpf{e+u,u}+r\qpf{e+v,v}+rs\qpf{e+u+v,u+v}\\[1ex]
 & = & \qf{1,r,rs,rst}_b\otimes[1,u]+\qf{r,rst}_b\otimes[1,v]
\end{array}$$
This shows that indeed $W_q(E/F)$ is already generated as
$WF$-module by $[1,u]$ and $[1,v]$, providing a new proof for 
Theorem \ref{multiquad}(iii).

\subsection{Generators in the mixed biquadratic case}
Let us assume that the degree $4$ extension $E/F$ is 
mixed biquadratic, so $E=F(\mu,\nu)$ with
$\mu^2=u$ and $\nu^2+\nu+v=0$ for some $u,v\in F^*$. In this case,
there are exactly two proper quadratic subextensions,
the separable quadratic extension $F(\wp^{-1}(v))$ and the
inseparable quadratic extension $F(\sqrt{u})$.  The only
nonhyperbolic forms of type (a) or type (b) in Theorem \ref{main}
are thus, up to isometry, $[1,v]\in WF\otimes [1,v]$ 
and $\qpf{u,w}\in \qf{1,u}_b\otimes W_q(F)$ for suitable $w\in F^*$.
So the last type to consider are 
generators $\qpf{f_C(e),\frac{d}{e^2}}$ of type (c).  We proceed
as before and write $E$ as a simple extension.  Put
$\gamma=\mu\nu$.  Then $E=F(\gamma)$ and the minimal polynomial of
$\gamma$ is $f(X)=X^4+uX^2+u^2v^2$ with cubic resolvent
$f_C(X)=X^3+uX^2$.  In $W_q(F)$, we then get with 
$d=u^2v^2$ and $r=\frac{uv}{e}$:
$$\begin{array}{rcl}
\qpf{f_C(e),\frac{d}{e^2}} & = & \qpf{e^3+e^2u,(\frac{uv}{e})^2}
= \qpf{e+u,\frac{uv}{e}}\\[1ex]
 & = & \qpf{\frac{uv}{e}(e+u),\frac{uv}{e}}=
\qpf{u(v+\frac{uv}{e}),\frac{uv}{e}}\\[1ex]
 & = & \qpf{u,\frac{uv}{e}}+u\qpf{v+\frac{uv}{e},\frac{uv}{e}}
=\qpf{u,r}+u\qpf{v+r,v}\\[1ex]
 & = & \qf{1,u}_b\otimes [1,r]+\qf{u,u(v+r)}_b\otimes [1,v]
\end{array}$$
and we recover Theorem \ref{multiquad}(ii) in the case
$n=1$ there, a result that is originally due to Ahmad \cite[Theorem 2.1]{a2}.

\subsection{Generators in the simple purely inseparable case}
Here, $E=F(\alpha)$ with $\alpha=\sqrt[4]{d}$ a root of $X^4+d$ where
$d\in F\setminus F^2$.  By Ahmad's result Theorem \ref{simpleinsep},
$W_q(E/F)$ is generated by $2$-fold quadratic Pfister forms of type
$\qpf{d,x}$ and $\qpf{x,dx^2y^2}$ with $x\in F^*$ and $y\in F^2(d)^*$. 
Forms of type $\qpf{d,x}$ are generators of type (b) in our list.
We now show how to express $\pi=\qpf{x,dx^2y^2}$ in terms of generators
of type (c) in our list.  

Let us write $y\in F^2(d)^*$ as $y=u^2+dv^2$ for some $u,v\in F$ with
$u\neq 0$ or $v\neq 0$.
Hence, in $W_q(F)$:
$$\qpf{x,dx^2y^2}=\qpf{x,dx^2(u^4+d^2v^4)}=
\qpf{x,dx^2u^4}+\qpf{x,d^3x^2v^4}\ .$$
If $u\neq 0$ put $s=\frac{1}{xu^2}$ and if $v\neq 0$ put
$t=\frac{1}{dxv^2}$.    If $u=0$ (and thus $v\neq 0$), we get
$$\begin{array}{rcl}
\pi & = & \qpf{x,dx^2u^4} = \qpf{x(\frac{1}{xu})^2,\frac{d}{s^2}}\\[1ex]
 & = & \qpf{s,\frac{d}{s^2}}= \qpf{s^3,\frac{d}{s^2}}\\[1ex]
 & = & \qpf{f_C(s),\frac{d}{s^2}}\ .
\end{array}$$
If $v=0$ (and thus $u\neq 0$), we get 
$$\begin{array}{rcl}
\pi & = & \qpf{x,d^3x^2v^4} = \qpf{x(d^3x^2v^4),\frac{d}{t^2}}\\[1ex]
 & = & \qpf{x^3d^3v^4(vt^3)^2,\frac{d}{t^2}}= \qpf{t,\frac{d}{t^2}}\\[1ex]
 & = & \qpf{f_C(t),\frac{d}{t^2}}\ .
\end{array}$$
Finally,  if both $u,v\neq 0$, we have
$\pi = \qpf{f_C(s),\frac{d}{s^2}}+\qpf{f_C(t),\frac{d}{t^2}}$.

\section{An application to Brauer kernels}\label{brauer}
Let $F$ be a field of characteristic $2$.  Recall that the 
central simple $F$-algebras of degree $2$ are exactly the
quaternion algebras $(a,b]$, $a,b\in F$, $b\neq 0$ generated by 
elements $e,f$ satisfying the relations $e^2=a$, $f^2+f=b$, $ef=(f+1)e$.
Such an algebra $(a,b]$ is a division algebra iff its norm form
$\qpf{a,b}$ is anisotropic, and $(a,b]\cong (a',b']$ iff
$\qpf{a,b}\cong\qpf{a',b'}$ (see, e.g., \cite[Prop.\ I.1.19]{b2}).

An Albert form $q$ is a $6$-dimensional nonsingular quadratic form
with $\Delta(q)=0\in F/\wp(F)$.  In particular, there exist
$\lambda,x,y\in F^*$, $u,v\in F$ such that
$\lambda q\cong [1,u+v]\perp x[1,u]\perp y[1,v]$.  In $W_qF$, we have
$\lambda q =\qpf{x,u}+\qpf{y,v}$, and the Clifford algebra $C(q)$  of
$q$ is Brauer equivalent to the biquaternion algebra $A=(x,u]\otimes (y,v]$
and it only depends on the similarity class of $q$.  Conversely,
given such a biquaternion algebra $A$, any nonsingular form $q$ of
dimension $6$ with trivial Arf-invariant that satisfies $C(q)=A\in \Br(F)$
will be called an Albert form for $A$.  Note that the index $\ind(A)$
will be $1$, $2$ or $4$.
The following well known
theorem is due to Jacobson \cite{j} (see also \cite{ms}).

\begin{theo}\label{jacobson}  
Let $q$ and $q'$ be Albert forms for the biquaternion
algebras $A$ and $A'$, respectively.  
\begin{enumerate}
\item[{\rm (i)}] $q$ is similar to $q'$ iff
$A\cong A'$.  
\item[{\rm (ii)}] $q$ is anisotropic iff $A$ is a division algebra, i.e.
$\ind (A)=4$.
\item[{\rm (iii)}] $i_W(q)=1$ iff $A=Q\in \Br(F)$ for a quaternion division algebra
$Q$, i.e. $\ind (A) =2$.
\item[{\rm (iv)}] $i_W(q)=3$ iff $A$ is split, i.e. $\ind(A)=1$.
\end{enumerate}  
\end{theo}

\begin{theo} Let $E/F$ be a quartic extension.

{\rm (i)}  Let $Q$ be a quaternion algebra over $F$
Then $Q\in \Br_2(E/F)$ iff $Q$ is of one of the following types:
\begin{enumerate}
\item[{\rm (a)}] $(h,g]$ for $h\in F^*$ and $g\in F$ such that 
$F(\wp^{-1}(g))\subset E$;
\item[{\rm (b)}] $(g,h]$ for $h\in F$ and $g\in F^*$ such that 
$F(\sqrt{g})\subset E$;
\item[{\rm (c)}] $(f_C(e),\frac{d}{e^2}]$ for $e\in F^*$ with 
$f_C(e)\neq 0$.
\end{enumerate}

{\rm (ii)} If $D$ is a nontrivial division algebra with $D\in \Br_2(E/F)$, then
either $D\cong Q$ or $D\cong Q_1\otimes Q_2$ where 
$Q$, $Q_1$, $Q_2$ are quaternion algebras of type (a), (b), or (c).
In particular, $\Br_2(E/F)$ is generated by such quaternion algebras.
\end{theo}
\begin{proof}
(i) Let $Q=(x,y]$.  By the above remarks, $Q\in\Br(E/F)$ iff 
$\qpf{x,y}\in W_q(E/F)$.  The result then follows readily from
Corollary \ref{2-fold}.

(ii) Let $D$ be a nontrivial division algebra with $D\in \Br_2(E/F)$.  Since
$[E:F]=4$, one necessarily has $\ind(D)=2$ or $\ind(D)=4$.
If $\ind(D)=2$, then $D$ is a quaternion algebra and the result follows
from (i).  If $\ind(D)=4$, then $D$ is a biquaternion algebra
by Albert's theorem \cite[p.\ 174]{alb}.  Let $\phi$ be an Albert form
associated with $D$. By Theorem \ref{jacobson}, $\phi$ is anisotropic
and $\phi\in W_q(E/F)$. 
By Corollary \ref{albert}, there exist 
quaternion algebras $Q_1$, $Q_2$ as in (i) with norm forms
$\pi_1$ and $\pi_2$, respectively, and $\lambda\in F^*$ such that
$\phi =\lambda(\pi_1+\pi_2)\in W_qF$.  But then $\phi$ is also
an Albert form for $Q_1\otimes Q_2$, and again by Theorem \ref{jacobson},
 we have $D\cong Q_1\otimes Q_2$.
\end{proof}

\end{document}